\documentclass[12pt,reqno]{amsart}
\usepackage[a4paper]{geometry}
\usepackage{amssymb,amsmath,amsthm,enumerate,verbatim,bbm}

\DeclareMathOperator{\Ker}{Ker}

\newcommand{\ac}{\text{ac}}
\newcommand{\ess}{\text{ess}}

\renewcommand\Im{\hbox{{\rm Im}}\,}

\newcommand{\abs}[1]{\lvert#1\rvert}
\newcommand{\Abs}[1]{\left\lvert#1\right\rvert}

\newcommand{\Hank}{H}

\newcommand{\bbR}{{\mathbb R}}
\newcommand{\bbC}{{\mathbb C}}

\newcommand{\bbZ}{{\mathbb Z}}

\newcommand{\calM}{\mathcal{M}}

\numberwithin{equation}{section}
\renewcommand{\[}{\begin{equation}}
\renewcommand{\]}{\end{equation}}

\theoremstyle{plain}
\newtheorem{theorem}{\bf Theorem}[section]
\newtheorem*{theorem*}{Theorem 1.1$'$}
\newtheorem{lemma}[theorem]{\bf Lemma}
\newtheorem{proposition}[theorem]{\bf Proposition}

\theoremstyle{definition}

\newtheorem*{definition*}{\bf Definition}

\theoremstyle{remark}
\newtheorem*{remark*}{\bf Remark}

\newcommand{\eps}{\varepsilon}

\newcommand{\1}{\mathbbm{1}}

\newcommand{\Sch}{{\mathbf{S}}}

\newcommand{\f}{\varphi}

\begin{document}

\title[Weighted Hankel operators]{Weighted integral Hankel operators with continuous spectrum}

\author{Emilio Fedele}
\address{Department of Mathematics, King's College London, Strand, London, WC2R~2LS, U.K.}
\email{emilio.fedele@kcl.ac.uk}

\author{Alexander Pushnitski}
\address{Department of Mathematics, King's College London, Strand, London, WC2R~2LS, U.K.}
\email{alexander.pushnitski@kcl.ac.uk}



\begin{abstract}
Using the Kato-Rosenblum theorem, we describe the absolutely continuous spectrum 
of a class of weighted integral Hankel operators in $L^2(\bbR_+)$. 
These self-adjoint operators generalise the explicitly diagonalisable operator with the integral 
kernel $s^\alpha t^\alpha(s+t)^{-1-2\alpha}$, where $\alpha>-1/2$.
Our analysis can be considered as an extension of J.~Howland's 1992 paper which 
dealt with the unweighted case, corresponding to $\alpha=0$. 
\end{abstract}

\date{\today}

\maketitle

\section{Introduction}\label{sec.a}

The aim of this paper is to consider some variants (perturbations) of the following simple integral operator:
\begin{gather}
A_\alpha: L^2(\bbR_+)\to L^2(\bbR_+),\quad \alpha>-1/2,
\notag
\\
(A_\alpha f)(t)=\int_0^\infty \frac{t^\alpha s^\alpha}{(s+t)^{1+2\alpha}}f(s)ds,
\quad f\in L^2(\bbR_+).
\label{a1}
\end{gather}
Since the integral kernel of $A_\alpha$ is homogeneous of degree $-1$, 
this operator can be explicitly diagonalised by the Mellin transform
$$
\calM f(\xi)=\frac1{\sqrt{2\pi}}\int_0^\infty t^{-\frac12+i\xi}f(t)dt, \quad \xi\in\bbR,
$$
which is a unitary map from $L^2(\bbR_+,dt)$ to $L^2(\bbR,d\xi)$. 
Mellin transform effects
a unitary transformation of $A_\alpha$ into the operator of multiplication by the function 
(here $\Gamma$ is the standard Gamma function)
$$
\bbR\ni\xi\mapsto
\frac{\abs{\Gamma(\frac12+\alpha+i\xi)}^2}{\Gamma(1+2\alpha)}
$$
in $L^2(\bbR,d\xi)$. 
The spectrum of $A_\alpha$ is given by the range of this function. 
Observe that this  function is even in $\xi$ and monotone increasing on $(-\infty,0)$; 
we denote its maximum, attained at $\xi=0$, by 
\[
\pi_\alpha=\frac{\Gamma(\frac12+\alpha)^2}{\Gamma(1+2\alpha)}.
\label{a1a}
\]
With this notation, we can summarise the above discussion by 
\begin{proposition}\label{prp.a1}
For $\alpha>-1/2$, the operator $A_\alpha$ of \eqref{a1} in $L^2(\bbR_+)$ is 
bounded and selfadjoint, and has a purely absolutely continuous (a.c.) spectrum of multiplicity two given by 
$$
\sigma_\ac(A_\alpha)=[0,\pi_\alpha].
$$
\end{proposition}
This includes the well-known case $\alpha=0$ of the Carleman operator; in this case $\pi_0=\pi$. 

In \cite{How2}, Howland considered integral Hankel operators on $L^2(\bbR_+)$ with kernels 
whose asymptotic behaviour is modelled on that of the Carleman operator. 
For a real-valued function $a=a(t)$, $t>0$ (we call it a \emph{kernel}),
let us denote by $\Hank(a)$ the Hankel operator in $L^2(\bbR_+)$ defined by 
$$
(\Hank(a)f)(t)=\int_0^\infty a(t+s)f(s)ds, \quad t>0.
$$
Howland considered kernels $a$ with the asymptotic behaviour 
\[
ta(t)\to 
\begin{cases}
a_0 & t\to0,
\\
a_\infty & t\to\infty.
\end{cases}
\label{a2}
\]
Among other things, in \cite{How2} he proved
\begin{theorem}\cite{How2}\label{thmB}
Let $a\in C^2(\bbR_+)$ have the asymptotic behaviour \eqref{a2} and satisfy the 
regularity conditions
$$
(ta(t))''(t)=
\begin{cases}
O(t^{-2+\eps}), & t\to0,
\\
O(t^{-2-\eps}), & t\to\infty,
\end{cases}
$$
with some $\eps>0$. Then the a.c. spectrum of 
$\Hank(a)$ is given by 
\[
\sigma_\ac(\Hank(a))=[0,\pi a_0]\cup[0,\pi a_\infty],
\label{a4}
\]
where each interval contributes multiplicity one to the spectrum.
\end{theorem}
We pause here to explain the convention that is used in \eqref{a4} and that will be
used in similar relations below. 
Relation \eqref{a4} means that the 
a.c. part of $\Hank(a)$ is unitarily equivalent to the direct sum
of the operators of multiplication by $\lambda$ in $L^2([0,\pi a_0],d\lambda)$ and 
in $L^2([0,\pi a_\infty],d\lambda)$. We also assume that if, for example, $a_0=0$, then 
the first term drops out of the union in \eqref{a4}; and that if, for example, $a_\infty<0$, then 
the interval $[0,\pi a_\infty]$ should be understood as $[\pi a_\infty,0]$.

Theorem~\ref{thmB} makes precise the intuition that for the Carleman operator $A_0$, corresponding to the kernel $a(t)=1/t$, 
both $t=0$ and $t=\infty$ are singular points and each of these points contributes multiplicity one 
to the spectrum. The aim of this paper is to show that the above intuition is also valid for operators 
$A_\alpha$ with all $\alpha>-1/2$. 
We do this by considering weighted Hankel operators. 
These operators generalise $A_\alpha$ in the same manner as the operators $\Hank(a)$ with kernels as in 
Theorem~\ref{thmB} generalise the Carleman operator $A_0$.

For a real-valued
kernel $a(t)$ and for a complex-valued function (we will call it a \emph{weight}) $w(t)$, $t>0$, 
we denote by $w\Hank(a)\overline{w}$ 
the weighted Hankel operator in $L^2(\bbR_+)$, given by 
$$
(w\Hank(a)\overline{w}f)(t)
=
\int_0^\infty w(t)a(t+s)\overline{w(s)}f(s)ds, 
\quad f\in L^2(\bbR_+).
$$
Under our assumptions below, this operator will be bounded. 
Since $a$ is assumed real-valued, the operator $w\Hank(a)\overline{w}$ 
is self-adjoint. Here and in what follows by a slight abuse of notation we use the same
symbol (in this case $w$) to denote both a function on $\bbR_+$ and the operator
of multiplication by this function in $L^2(\bbR_+)$. 

We fix $\alpha>-1/2$ and consider $a$, $w$ with the asymptotic behaviour
\[
t^{1+2\alpha}a(t)\to
\begin{cases}
a_0 & t\to0,
\\
a_\infty & t\to\infty,
\end{cases}
\quad
t^{-\alpha}w(t)
\to
\begin{cases}
b_0 & t\to0,
\\
b_\infty & t\to\infty.
\end{cases}
\label{a6}
\]
The aim of this paper is to prove
\begin{theorem}\label{thm.a3}
Fix $\alpha>-1/2$. 
Let $a\in C^2(\bbR_+)$ be a real-valued kernel such that 
for some $a_0,a_\infty\in\bbR$ and for some $\eps>0$, we have
\begin{align}
\frac{d^m}{dt^m}(t^{1+2\alpha}a(t)-a_0)
&=O(t^{-m+\eps}),\quad t\to0,
\label{a7b}
\\
\frac{d^m}{dt^m}(t^{1+2\alpha}a(t)-a_\infty )
&=O(t^{-m-\eps}),\quad t\to\infty,
\label{a7c}
\end{align}
with $m=0,1,2$. 
Assume further that the complex valued weight $w(t)$ is such that 
$t^{-\alpha}w(t)$ is bounded on $\bbR_+$ and 
for some $b_0,b_\infty\in\bbC$, 
\[
\int_0^1\Abs{\abs{w(t)}^2t^{-2\alpha}-\abs{b_0}^2}t^{-1}dt<\infty,
\quad
\int_1^\infty\Abs{\abs{w(t)}^2t^{-2\alpha}-\abs{b_\infty}^2}t^{-1}dt<\infty.
\label{a7a}
\]
Then the a.c. spectrum of $w\Hank(a)\overline{w}$ is 
given by 
$$
\sigma_\ac(w\Hank(a)\overline{w})
=
[0,\pi_\alpha a_0\abs{b_0}^2]\cup[0,\pi_\alpha a_\infty\abs{b_\infty}^2],
$$
where each interval contributes multiplicity one to the spectrum. 
\end{theorem}

\begin{remark*}
\begin{enumerate}
\item
Howland in \cite{How2} uses Mourre's estimate and proves also the absence
of singular continuous spectrum in the framework of Theorem~\ref{thmB}. 
Here we use the trace class method of scattering theory. This method  
is technically simpler to use but it gives
no information on the singular continuous spectrum. 

\item
Conditions on $a$ and $w$ in Theorem~\ref{thm.a3} are far from being sharp. 
For example, 
it is not difficult to relax conditions \eqref{a7b}, \eqref{a7c} by replacing 
$t^{\pm \eps}$ by $\abs{\log t}^{-1-\eps}$, see \cite{PYa2} for a related calculation.

\item
Howland's results of \cite{How2} for unweighted Hankel operators were 
extended in \cite{PYa} to kernels $a(t)$ with more complicated (oscillatory) asymptotic behaviour at $t\to\infty$. 

\item
An important precursor to Howland's work \cite{How2} was Power's analysis \cite{Pow} of the essential spectrum 
of Hankel operators with piecewise continuous symbols. In this context we note that the essential spectrum of 
the weighted Hankel operators considered in Theorem~\ref{thm.a3} is easy to describe. 
By following the method of proof of this theorem and using Weyl's theorem on the preservation
of the essential spectrum under compact perturbations instead of the Kato-Rosenblum theorem, one can check
that if both $t^{1+2\alpha}a(t)$ and $t^{-\alpha}w(t)$ are bounded and satisfy the asymptotic relation \eqref{a6},
then the essential spectrum of $w\Hank(a)\overline{w}$ is given by the union of the intervals
$$
\sigma_\ess(w\Hank(a)\overline{w})
=
[0,\pi_\alpha a_0\abs{b_0}^2]\cup[0,\pi_\alpha a_\infty\abs{b_\infty}^2].
$$

\item
Boundedness and Schatten class conditions for weighted Hankel operators
with the power weights $w_\alpha(t)=t^\alpha$ have been studied by several authors;
see e.g. \cite{Rochberg,JP} and the references in \cite[Section 2]{AP}.

\item
In \cite{KS}, interesting non-trivial discrete analogues of the operators $A_\alpha$ are analysed. 
These operators act in $\ell^2(\bbZ_+)$ and are formally defined as infinite matrices with entries of the form
\[
w(j)a(j+k)w(k), \quad j,k\in\bbZ_+.
\label{a8}
\]
For each $\alpha>-1/2$, 
the authors of \cite{KS} describe some families of sequences $\{a(j)\}$ and $\{w(j)\}$ with the asymptotic
behaviour 
$$
j^{1+2\alpha}a(j)\to1, \quad j^{-\alpha}w(j)\to1,\quad j\to\infty,
$$
for which the operators \eqref{a8} are explicitly diagonalised. 
It turns out that the spectrum of each of these operators is purely a.c., has 
multiplicity one and coincides with the interval $[0,\pi_\alpha]$, where $\pi_\alpha$ is the same as in \eqref{a1a}. 
\end{enumerate}
\end{remark*}

\section{Proof of Theorem~\ref{thm.a3}}

\subsection{Outline of the proof}
Let $a$, $w$ be as in Theorem~\ref{thm.a3}.
First we identify two suitable ``model'' kernels $\f_0$ and $\f_\infty$
in $C^\infty(\bbR_+)$ such that $\f_0(t)+\f_\infty(t)=t^{-1-2\alpha}$ and
\[
\frac{d^m}{dt^m}
\f_0(t) =O(e^{-t/2}), \quad t\to\infty,
\quad\text{ and }\quad
\frac{d^m}{dt^m}
\f_\infty(t)=O(1), \quad t\to0
\label{a8a}
\]
for all $m\geq0$. 
Then we write the kernel $a$ as
$$
a(t)=a_0\f_0(t)+a_\infty \f_\infty(t)+\text{error},
$$
where the error term is negligible in a suitable sense both as $t\to0$ and as $t\to\infty$. 
Similarly, we write
$$
\abs{w(t)}^2=\abs{b_0}^2\1_0(t)t^{2\alpha}+\abs{b_\infty}^2\1_\infty(t)t^{2\alpha}+\text{error},
$$
where $\1_0$ and $\1_\infty$ are the characteristic functions of the intervals $(0,1)$ and $(1,\infty)$ respectively
and the error term is again negligible in a suitable sense. 
With these representations, denoting $w_\alpha(t)=t^{\alpha}$, we write
\[
w\Hank(a)w
=
a_0\abs{b_0}^2 \1_0w_\alpha\Hank(\f_0)w_\alpha\1_0
+
a_\infty\abs{b_\infty}^2 \1_\infty w_\alpha\Hank(\f_\infty)w_\alpha\1_\infty
+\text{error}
\label{a10}
\]
and prove that the error term here is a trace class operator. 
By the Kato-Rosenblum theorem (see e.g. \cite[Theorem XI.8]{RS3}), this reduces the problem to the description of the 
a.c. spectrum of the sum of the first two operators in the right side of \eqref{a10}. 
Observe that these two operators act in the orthogonal subspaces
$L^2(0,1)$ and $L^2(1,\infty)$. 
This reduces the problem to identifying the a.c. spectra of
\[
\1_0w_\alpha\Hank(\f_0)w_\alpha\1_0
\quad \text{ and }\quad
\1_\infty w_\alpha\Hank(\f_\infty) w_\alpha\1_\infty.
\label{a11}
\]
We are unable to identify the spectra of these operators directly and therefore we resort to 
the following trick. 
We observe that the operator $A_\alpha$, whose spectrum is given by Proposition~\ref{prp.a1},
can also be represented in the form \eqref{a10} with $a_0\abs{b_0}^2=a_\infty\abs{b_\infty}^2=1$. 
This allows us to conclude that the a.c. spectrum of each of the two operators in \eqref{a11} 
coincides with $[0,\pi_\alpha]$ and has multiplicity one. 
Now we can go back to \eqref{a10} and finish the proof.

\subsection{Factorisation of $A_\alpha$}
For $\alpha>-1/2$, let $L_\alpha$ be the integral operator in $L^2(\bbR_+)$ given by 
\[
(L_\alpha f)(t)
=
\frac{1}{\sqrt{\Gamma(1+2\alpha)}}
\int_0^\infty
t^\alpha s^\alpha e^{-st} f(s)ds, \quad t>0.
\label{b10}
\]
The boundedness of $L_\alpha$ is easy to establish by the Schur test. 
It is evident that $L_\alpha$ is self-adjoint. 
A direct calculation gives the identity
$$
A_\alpha=L_\alpha^2. 
$$
This factorisation is an important technical ingredient of the proof. 

\subsection{Trace class properties of auxiliary operators}

\begin{lemma}\label{lma.b1}
Let $L_\alpha$ be the operator \eqref{b10} and 
let $u$ be a locally integrable function on $\bbR_+$. 
Then the operator $uL_\alpha$ is in the Hilbert-Schmidt class if and only if
$$
\int_0^\infty \abs{u(t)}^2\frac{dt}{t}<\infty.
$$
\end{lemma}
\begin{proof}
A direct evaluation of the Hilbert-Schmidt norm:
$$
\frac1{\Gamma(1+2\alpha)}
\int_0^\infty \int_0^\infty s^{2\alpha}t^{2\alpha}e^{-2ts}\abs{u(t)}^2dt\,ds
=
2^{-1-2\alpha}\int_0^\infty \abs{u(t)}^2\frac{dt}{t}.
\qedhere
$$
\end{proof}

A necessary and sufficient condition is known (see \cite{Rochberg}) 
for $w_\alpha \Hank(g)w_\alpha$ to belong to trace class in terms of $g$ being in a certain Besov class. 
For our purposes it suffices to use a simple sufficient condition expressed in elementary terms. 
\begin{lemma}\label{lma.b2}
Let $g\in C^2(\bbR_+)$ be such that for some $\eps>0$ and for $m=0,1,2$, one has
$$
\frac{d^m}{dt^m}(t^{1+2\alpha}g(t))
=
\begin{cases}
O(t^{-m+\eps}), & t\to0,
\\
O(t^{-m-\eps}), & t\to\infty.
\end{cases}
$$
Then $w_\alpha\Hank(g)w_\alpha$ is trace class. 
\end{lemma}
\begin{proof}
Lemma 2 in \cite{Rochberg} asserts that $w_\alpha\Hank(g)w_\alpha$ is trace class
if the function $k(t)=t^{2+2\alpha}g(t)$ satisfies the condition
$$
\int_{-\infty}^\infty \int_0^\infty \abs{\hat k(x+iy)}dy\,dx<\infty,
$$
where 
$$
\hat k(\zeta)=\int_{0}^\infty k(t)e^{i\zeta t}dt, \quad \zeta=x+iy, \quad y>0.
$$
Let us check that this condition is satisfied under our hypothesis on $g$. 
First note that under our hypothesis, we have
\[
k^{(m)}(t)=O(t^{1-m+\eps}), \quad t\to0, 
\qquad
k^{(m)}(t)=O(t^{1-m-\eps}),
\quad t\to\infty.
\label{b11}
\]
Next, 
integrating by parts once and twice in the expression for $\hat k$, we get
$$
\hat k(\zeta)
=
-\frac1{i\zeta}\int_0^\infty k'(t)e^{i\zeta t}dt
=
\frac1{(i\zeta)^2}\int_0^\infty k''(t)e^{i\zeta t}dt, \quad \Im \zeta>0, 
$$
and therefore we have the estimates
\[
\abs{\hat k(\zeta)}\leq\frac1{\abs{\zeta}}\int_0^\infty \abs{k'(t)}e^{-yt}dt, 
\quad
\abs{\hat k(\zeta)}\leq\frac1{\abs{\zeta}^2}\int_0^\infty \abs{k''(t)}e^{-yt}dt
\label{b9}
\]
for $\zeta=x+iy$. For $\abs{\zeta}\leq1$ we use the first one of these estimates, 
which together with \eqref{b11} yields
$$
\abs{\hat k(\zeta)}
\leq 
\frac{C}{\abs{\zeta}}\int_0^1 t^\eps e^{-yt}dt+
\frac{C}{\abs{\zeta}}\int_1^\infty t^{-\eps} e^{-yt}dt
\leq C\frac{1+y^{-1+\eps}}{\abs{\zeta}}.
$$
The right side here is integrable in the domain $\abs{\zeta}<1$, $\Im \zeta>0$, if $0<\eps<1$.

For $\abs{\zeta}>1$ we use the second estimate in \eqref{b9}, which yields
$$
\abs{\hat k(\zeta)}
\leq 
\frac{C}{\abs{\zeta}^2}\int_0^1 t^{-1+\eps} e^{-yt}dt+
\frac{C}{\abs{\zeta}^2}\int_1^\infty t^{-1-\eps} e^{-yt}dt
\leq C\frac{y^{-\eps}+e^{-y}}{\abs{\zeta}^2},
$$
and again the right side is integrable in the domain $\abs{\zeta}>1$, $\Im \zeta>0$, if $0<\eps<1$. 
\end{proof}

The following lemma allows us to get rid of the cross terms that are hidden in
the error term in \eqref{a10}.
\begin{lemma}\label{lma.b3}
The operators $\1_0L_\alpha\1_0$ and $\1_\infty L_\alpha\1_\infty$ are trace class. 
Further, the operators $\1_0 A_\alpha\1_\infty$ and $\1_\infty A_\alpha\1_0$ are trace class. 
\end{lemma}
\begin{proof}
Let us prove the first statement. 
We will regard $\1_0L_\alpha\1_0$ as acting on $L^2(0,1)$ and $\1_\infty L_\alpha\1_\infty$ as acting on $L^2(1,\infty)$. 
Consider the unitary operators 
\begin{align*}
&U_+:L^2(1,\infty)\to L^2(\bbR_+), \quad
(U_+f)(x)=e^{x/2}f(e^x), \quad x>0,
\\
&U_-:L^2(0,1)\to L^2(\bbR_+), \quad
(U_-f)(x)=e^{-x/2}f(e^{-x}), \quad x>0.
\end{align*}
A straightforward calculation shows that
$$
U_+\1_\infty L_\alpha \1_\infty U_+^*=\Hank(\psi_+)
\quad\text{ and }\quad
U_-\1_0 L_\alpha \1_0 U_+^*=\Hank(\psi_-),
$$
where the kernels $\psi_\pm$ are given by
$$
\psi_+(t)=e^{t(\alpha+1/2)}e^{-e^{t}}, 
\quad
\psi_-(t)=e^{-t(\alpha+1/2)}e^{-e^{-t}},
\quad t>0.
$$
As both functions $\psi_\pm$ are Schwartz class, 
using Lemma~\ref{lma.b2} we find that the unweighted Hankel operators  $\Hank(\psi_\pm)$ 
are trace class. 

To prove the second statement of the lemma, we write $1=\1_0+\1_\infty$ and use the factorisation 
$A_\alpha=L_\alpha^2$ to obtain
$$
\1_0A_\alpha\1_\infty
=
\1_0L_\alpha^2\1_\infty
=
\1_0L_\alpha(\1_0+\1_\infty)L_\alpha\1_\infty
=
(\1_0L_\alpha\1_0) L_\alpha\1_\infty
+
\1_0L_\alpha(\1_\infty L_\alpha\1_\infty).
$$
Now observe that both terms in the right side are trace class by the first part of the lemma.
Thus, $\1_0A_\alpha\1_\infty$ is trace class and by a similar reasoning  $\1_\infty A_\alpha\1_0$
is also trace class. 
\end{proof}

We note that a more careful analysis of the kernels $\psi_\pm$ shows that the operators
$\1_0L_\alpha\1_0$ and $\1_\infty L_\alpha\1_\infty$ belong to the Schatten class $\Sch_p$ 
for any $p>0$. 

\subsection{Kernels $\f_0$ and $\f_\infty$}
Recall the notation $w_\alpha(t)=t^{\alpha}$. 
By a direct calculation of the integral kernels, we have
$$
L_\alpha\1_\infty L_\alpha
=
w_\alpha\Hank(\f_0)w_\alpha,
\quad
L_\alpha\1_0 L_\alpha
=
w_\alpha\Hank(\f_\infty)w_\alpha,
$$
with
$$
\f_0(t)=\frac1{\Gamma(1+2\alpha)}\int_1^\infty x^{2\alpha}e^{-xt}dx,
\quad
\f_\infty(t)=\frac1{\Gamma(1+2\alpha)}\int_0^1 x^{2\alpha}e^{-xt}dx.
$$
Using the integral representation for the Gamma function, we obtain
$$
\f_0(t)+\f_\infty(t)=t^{-1-2\alpha}, \quad t>0.
$$
Further, it is straightforward to see that 
the estimates \eqref{a8a} hold true for all $m\geq0$. 
The following lemma gives a description of the spectra of the two operators \eqref{a11}. 

\begin{lemma}\label{lma.b4}
We have
\[
\sigma_\ac(\1_0 L_\alpha\1_\infty L_\alpha \1_0)
=
\sigma_\ac(\1_\infty L_\alpha\1_0 L_\alpha \1_\infty)
=
[0,\pi_\alpha],
\label{b3}
\]
with multiplicity one in both cases.
\end{lemma}
\begin{proof}
First let us consider the operators $\1_0A_\alpha\1_0$ and $\1_\infty A_\alpha\1_\infty$.
We claim that these two operators are unitarily equivalent to each other. 
Indeed, let 
$$
U:L^2(\bbR_+)\to L^2(\bbR_+), 
\qquad
(Uf)(t)=(1/t)f(1/t), \quad t>0.
$$
Then it is easy to see that $U$ is unitary and $UA_\alpha  U^*=A_\alpha$. 
It follows that 
\[
U\1_\infty A_\alpha \1_\infty U^*=\1_0A_\alpha\1_0.
\label{b12}
\]
Next, write
$$
A_\alpha
=
\1_0 A_\alpha \1_0+
\1_\infty A_\alpha \1_\infty+
(\1_\infty A_\alpha \1_0+
\1_0 A_\alpha \1_\infty).
$$
By Lemma~\ref{lma.b3}, the two cross terms in brackets here are trace class;
thus, we can apply the Kato-Rosenblum theorem. 
Recalling Proposition~\ref{prp.a1}, we obtain that the a.c. spectrum of the sum 
\[
\1_0 A_\alpha \1_0+
\1_\infty A_\alpha \1_\infty
\label{b1}
\]
is $[0,\pi_\alpha]$ with multiplicity two. 
Now observe that the two operators in \eqref{b1} act in orthogonal 
subspaces $L^2(0,1)$ and $L^2(1,\infty)$ and, by \eqref{b12},
they are unitarily equivalent to each other. 
Thus, we obtain
$$
\sigma_\ac(\1_0 A_\alpha \1_0)=\sigma_\ac(\1_\infty A_\alpha \1_\infty)=[0,\pi_\alpha],
$$
with multiplicity one in both cases. 

Finally, write
$$
\1_0A_\alpha\1_0
=
\1_0L_\alpha^2\1_0
=
\1_0L_\alpha\1_\infty L_\alpha\1_0+\1_0 L_\alpha\1_0 L_\alpha\1_0.
$$
By Lemma~\ref{lma.b3}, the second term in the right side here is trace class. 
Thus, by the Kato-Rosenblum theorem, we obtain 
$$
\sigma_\ac(\1_0L_\alpha\1_\infty L_\alpha\1_0)=[0,\pi_\alpha],
$$
with multliplicity one, which gives
the description of the a.c. spectrum of the first operator in \eqref{b3}. 
The second operator is considered in the same way.
\end{proof}

\subsection{Concluding the proof}

First we prove an intermediate statement. We denote $v(t)=w(t)t^{-\alpha}$
and use the notation $\Sch_1$ for the trace class.  
\begin{lemma}\label{lma.b5}
Under the hypothesis of Theorem~\ref{thm.a3}, we have
\begin{align}
\sigma_\ac(v\1_0 L_\alpha \1_\infty L_\alpha \1_0 \overline{v})
&=
[0,\pi_\alpha \abs{b_0}^2],
\label{b4}
\\
\sigma_\ac(v\1_\infty L_\alpha \1_0 L_\alpha \1_\infty \overline{v})
&=
[0,\pi_\alpha \abs{b_\infty}^2],
\label{b5}
\end{align}
with multiplicity one in both cases. 
\end{lemma}
\begin{proof}
We prove the first relation \eqref{b4}; the second relation is proven in a similar way.
First we write 
$$
v\1_0 L_\alpha \1_\infty L_\alpha \1_0 \overline{v}
=TT^*, \quad 
T=
v\1_0 L_\alpha \1_\infty
$$
and recall that for any bounded operator $T$, the operators 
$(TT^*)|_{(\Ker TT^*)^\perp}$ and $(T^*T)|_{(\Ker T^*T)^\perp}$ 
are unitarily equivalent. 
Thus, it suffices to describe the a.c. spectrum of the operator
$$
T^*T
=
\1_\infty L_\alpha\abs{v}^2 \1_0 L_\alpha \1_\infty.
$$

Next, by the hypothesis \eqref{a7a}, we can write
$$
\abs{v(t)}^2=\abs{b_0}^2+q_1(t)q_2(t), \quad \text{ with } \int_0^1 \frac{\abs{q_1(t)}^2+\abs{q_2(t)}^2}{t}dt<\infty.
$$
This yields 
$$
\1_\infty L_\alpha\abs{v}^2 \1_0 L_\alpha \1_\infty
=
\abs{b_0}^2
\1_\infty L_\alpha \1_0 L_\alpha \1_\infty
+
(\1_\infty L_\alpha\1_0 q_1)(q_2 \1_0 L_\alpha \1_\infty).
$$
By Lemma~\ref{lma.b1}, both operators in brackets here are Hilbert-Schmidt. 
It follows that the product of these operators is trace class, i.e. 
$$
\1_\infty L_\alpha\abs{v}^2 \1_0 L_\alpha \1_\infty
=
\abs{b_0}^2
\1_\infty L_\alpha \1_0 L_\alpha \1_\infty+T, \quad T\in\Sch_1.
$$
Lemma~\ref{lma.b4} gives the description of the a.c. spectrum of the first term 
in the right side here. 
Now an application of the Kato-Rosenblum theorem gives 
$$
\sigma_\ac(\1_\infty L_\alpha\abs{v}^2 \1_0 L_\alpha \1_\infty)
=
[0,\pi_\alpha\abs{b_0}^2],
$$
with multiplicity one. 
This yields \eqref{b4}. 
\end{proof}

\begin{proof}[Proof of Theorem~\ref{thm.a3}]
We would like to establish the representation 
\[
w\Hank(a)\overline{w}
=
a_0v(\1_0 L_\alpha\1_\infty L_\alpha\1_0)\overline{v}
+
a_\infty v(\1_\infty L_\alpha\1_0 L_\alpha\1_\infty)\overline{v}+T, 
\quad T\in\Sch_1. 
\label{b6}
\]
Observe that the first two operators in the right side act in orthogonal subspaces
and their a.c. spectra are described by Lemma~\ref{lma.b5}. 
Thus, applying the Kato-Rosenblum theorem, we
will have the required result as soon as the representation \eqref{b6} is proven. 

As a first step, let us write 
$$
a(t)=a_0 \f_0(t)+a_\infty \f_\infty(t)+g(t), \quad t>0,
$$
and examine the error term $g$. We have, using $\f_0(t)+\f_\infty(t)=t^{-1-2\alpha}$, 
\begin{multline*}
t^{1+2\alpha}g(t)
=
(t^{1+2\alpha}a(t)-a_0)+a_0(1-t^{1+2\alpha}\f_0(t))-a_\infty t^{1+2\alpha}\f_\infty(t)
\\
=
(t^{1+2\alpha}a(t)-a_0)+(a_0-a_\infty)t^{1+2\alpha}\f_\infty(t).
\end{multline*}
Thus, by the hypothesis \eqref{a7b} and by the second estimate in \eqref{a8a}, 
we obtain
$$
\frac{d^m}{dt^m}(t^{1+2\alpha}g(t))
=
O(t^{-m+\eps})+O(t^{-m+1+2\alpha})=O(t^{-m+\eps'}), \quad \eps'=\min\{\eps,1+2\alpha\},
$$
as $t\to0$.
Similarly, we have 
$$
t^{1+2\alpha}g(t)=(t^{1+2\alpha}a(t)-a_\infty)+(a_\infty-a_0)t^{1+2\alpha}\f_0(t), 
$$
and so, by the hypothesis \eqref{a7c} and by the first estimate in \eqref{a8a}, 
we get
$$
\frac{d^m}{dt^m}(t^{1+2\alpha}g(t))
=
O(t^{-m-\eps}), \quad t\to\infty.
$$
Thus, $g$ satisfies the hypothesis of Lemma~\ref{lma.b2} and so we obtain 
$$
w_\alpha\Hank(g)w_\alpha\in\Sch_1, \quad\text{ and so}\quad
w\Hank(g)\overline{w}\in\Sch_1.
$$
This gives the intermediate representation
\begin{multline}
w\Hank(a)\overline{w}
=
a_0w\Hank(\f_0)\overline{w}
+
a_\infty w\Hank(\f_\infty)\overline{w}
+T'
\\
=a_0vL_\alpha\1_\infty L_\alpha\overline{v}
+
a_\infty vL_\alpha\1_0 L_\alpha\overline{v}
+T', \quad T'\in\Sch_1. 
\label{b7}
\end{multline}
Consider the first term in the right side of \eqref{b7}.
We can write 
$$
L_\alpha\1_\infty L_\alpha
=
\1_0L_\alpha\1_\infty L_\alpha\1_0
+
(\1_\infty L_\alpha\1_\infty L_\alpha\1_\infty
+
\1_\infty L_\alpha\1_\infty L_\alpha\1_0
+
\1_0 L_\alpha\1_\infty L_\alpha\1_\infty).
$$
By Lemma~\ref{lma.b3}, all terms in brackets here are trace class operators, and so we obtain 
$$
vL_\alpha\1_\infty L_\alpha\overline{v}
-
v\1_0L_\alpha\1_\infty L_\alpha\1_0\overline{v}
\in\Sch_1.
$$
In a similar way, we obtain
$$
vL_\alpha\1_0 L_\alpha\overline{v}
-
v\1_\infty L_\alpha\1_0 L_\alpha\1_\infty \overline{v}
\in\Sch_1.
$$
Substituting this back into \eqref{b7}, we arrive at \eqref{b6}.
\end{proof}

\end{document}